\documentclass[11pt, reqno, oneside, notitlepage]{amsart}

\usepackage[a4paper, total={6in, 9.1in}]{geometry}

\usepackage{amsmath,amscd}
\usepackage{amssymb}
\usepackage{amsthm}
\usepackage{comment}
\usepackage{graphicx}
\usepackage{epstopdf} 
\usepackage{mathrsfs}
\usepackage{cite}

\usepackage{bm}
\usepackage{hyperref}
\usepackage{xcolor}
\hypersetup{
	colorlinks,
	linkcolor={blue},
	urlcolor={blue},
	citecolor={red}
}

\theoremstyle{plain}
\newtheorem{thm}{Theorem}[section]
\newtheorem{prop}{Proposition}[section]

\newtheorem{lem}[prop]{Lemma}

\newtheorem{assum}[prop]{Assumption}
\newtheorem{defi}[prop]{Definition}
\newtheorem{rmk}[prop]{Remark}

\newtheorem*{proposition*}{Proposition}

\numberwithin{equation}{section}

\newcommand {\p} {\partial}

\def\div{\text{div}}

\title[Inverse boundary problem for Mean Field Games]{On an inverse boundary problem for Mean Field Games }
\author[H. Liu]{Hongyu Liu}
\address{Department of Mathematics, City University of Hong Kong, Kowloon, Hong Kong SAR, China}
\email{hongyu.liuip@gmail.com, hongyliu@cityu.edu.hk}

\author[S. Zhang]{Shen Zhang}
\address{Department of Mathematics, City University of Hong Kong, Kowloon, Hong Kong SAR, China}
\email{szhang347-c@my.cityu.edu.hk}

\begin{document}
	\maketitle
	
	\begin{abstract}
	In this paper, we propose and study an inverse boundary problem for the mean field games (MFGs) governed by the first-order master equation in a bounded domain. We establish the unique identifiability result by showing that the running cost function within any given proper (state) space-time subdomain is uniquely determined by the MFG boundary data of this subdomain. There are several salient features of our study. The MFG system couples two nonlinear parabolic PDEs with one forward in time and the other one backward in time, and moreover there is a probability measure constraint. These make the inverse problem study new to the literature and highly challenging. We develop an effective and efficient scheme in tackling the inverse problem including the high-order variation in the probability space around a nontrivial distribution, and the construction of a novel class of CGO solutions to serve as the ``probing modes". Our study opens up a new field of research on inverse problems for MFGs.

%
	\end{abstract}

	\tableofcontents
	\section{Introduction}
\subsection{Problem setup and background}

	Mean Field Games (abbreviated as MFGs in what follows) are differential games involving non-atomic players, where one aims to study the behaviours of a large population of symmetric agents as the number of agents goes to infinity. They provide quantitative modelling of the macroscopic behaviours of the agents who wish to minimise a certain cost. The theory of MFGs was pioneered by Caines-Huang-Malhame \cite{HCM06,HCM071,HCM072,HCM073} and Lasry-Lions  \cite{LL06a, LL06b, LL07a, Lions} independently in 2006, and since then has received considerable attention and increasing studies int the literature. We refer to the books \cite{17} and the lecture notes \cite{CarPor} as well as the references cited therein for more related discussions. One of the main features of an MFG is that there is an adversarial regime and in this so-called monotone regime, the Nash equilibrium exists and is unique. The Nash equilibria of the MFGs can be characterized by the so-called master equation, which is a system of PDEs if the state space of the players is continuous. Next, we introduce the coupled PDE system which forms the forward model problem for our subsequent inverse problem study. 
	
Let $\mathbb{R}^n$ be the Euclidean space with $n\in\mathbb{N}$, and $\Omega'\subset\mathbb{R}^n$ be a bounded Lipschitz domain, which signifies the state space. Let $x\in\Omega'$ be the state variable and $t\in [0, \infty)$ be the time variable.  Let $\mathcal{P}$ stand for the set of Borel probability measures on $\mathbb{R}^n$, and  $\mathcal{P}(\Omega')$ stand for the set of Borel probability measures on $\Omega'$. Let $m\in\mathcal{P}(\Omega')$ denote the population distribution of the agents and $u(x, t):\Omega'\times [0, T]\mapsto \mathbb{R}$ denote the value function of each player. Here, $T\in\mathbb{R}_+$ signifies the terminal time in what follows. The MFG system for our study is introduced as follows:
\begin{equation}\label{main}
	\left\{
	\begin{array}{ll}
		\displaystyle{-\partial_t u(x,t) -\Delta u(t,x)+\frac{1}{2}|\nabla u(x,t)|^2-F(x, m(x,t))=0} &  {\rm{in}}\ Q',\medskip\\
		\displaystyle{\partial_tm(x,t)-\Delta m(x,t)-{\rm div} \big(m(x,t) \nabla u(x,t)\big)=0} & {\rm{in}}\ Q',\medskip\\
		\p_{\nu} u(x,t)=\p_{\nu} m(x,t)=0 & {\rm{on}}\ \Sigma',\medskip\\
		u(x,T)=G(x,m(x,T)),\ m(x,0)=m_0(x) & {\rm{in}}\  \Omega',\medskip
	\end{array}
	\right.
\end{equation}
where $\Delta$ and ${\rm div}$  are the Laplacian and divergent operators with respect to the $x$-variable, respectively; and 
$\Sigma':=\p\Omega'\times[0,T]$ , $Q':=\overline{\Omega'}\times[0,T]$ and $\nu$ is the exterior unit normal to $\partial\Omega'$. In \eqref{main}, $F:\Omega'\times\mathcal{P}(\Omega')\mapsto\mathbb{R} $ is the running cost function which signifies the interaction
between the agents and the population; $m_0$ is the initial population distribution
and $G:\Omega'\times\mathcal{P}(\Omega')\mapsto\mathbb{R}$ signifies the terminal cost. 

The MFG system \eqref{main} is also known as the first-order master equation which corresponds to the case that the volatility of the common noise among small players is vanishing \cite{CCP}. The master equation can be understood as an optimal nonlinear transport in the space of probability measures. The single player's value function $u$ satisfies the Hamilton-Jacobi-Bellman (HBJ) equation, namely the first equation in \eqref{main}, where the Hamiltonian is the canonical quadratic form $|\nabla u|^2/2$. The distribution law of the population $m$ is described by the Kolmogorov-Fokker-Planck (KFP) equation, namely the second equation in \eqref{main}.  When the  players have to control a process in a bounded domain $\Omega'\subset\mathbb{R}^n$ with reflexion on the boundary of $\Omega'$, it is constrained with the  Neumann boundary conditions. Many economic and financial models are described by this system; see for instance, the models in \cite{model}. It is also interesting to note that the HBJ equation is given backward in time, whereas the KFP equation is forward in time. In this paper, we focus on the case that $F$ and $G$ depend on $m$ locally. Due to this fact, we can define
\begin{equation}\label{eq:distr1}
\mathcal{O}:=\{ m:\Omega'\to [0,\infty) \ \ |\ \ \int_{\Omega'} m\, dx =1 \}.
\end{equation}
In other words, if $m\in \mathcal{O}$, then it is the density of a distribution in $\Omega'$. It can directly verified from \eqref{main} that if the initial distribution $m_0\in\mathcal{O}$, then $m(\cdot; t)\in\mathcal{O}$ for any subsequent time $t$. It is clear that by scaling, the total population $1$ in \eqref{eq:distr1} can be replaced by any positive number. 

The running cost $F$ is one of the key parameters (functions) in the MFG system. However, in practice the running cost is often  unknown or only partially known for the agents. This motivates us to consider the inverse problem of determining the running cost function $F$ by indirect measurement/knowledge of the MFG system that results from optimal actions. To that end, we next introduce the measurement/observation dataset for our inverse problem study. Let $\Omega$ be a given closed proper subdomain of $\Omega'$ with a smooth boundary $\partial\Omega$ and $Q:=\Omega\times [0, T]$, $\Sigma:=\partial\Omega\times[0,T].$
We define 
\begin{equation}\label{eq:meop0}
\mathcal{N}_F(m_0):=\Big( u(x, t), m(x, t)\Big)_{(x, t)\in \partial Q}, 
\end{equation}
where
$(u, m)$ is the (unique) pair of solutions to the MFG system \eqref{main} associated with the initial population distribution $m(x, 0)=m_0(x)$. Noting that $\partial Q=\Omega\times \{0, T\}\cup \Sigma$, we clearly have that
\begin{equation}\label{eq:meop1}
\mathcal{N}_F(m_0):=\Big(\big(u(x,0),m(x,T)\big)\Big|_{x\in\Omega}, \big(u(x,t), m(x,t)\big)\Big|_{(x,t)\in\Sigma}\Big). 
\end{equation}
 $\mathcal{N}(m_0)$ encodes the (state) space-time boundary data of $u$ and $m$ in the subdomain $\Omega\times[0, T]$ associated with a given initial population distribution $m_0$. The inverse problem that we aim to investigate can be formulated as follows:
\begin{equation}\label{eq:ip1}
\mathcal{N}_{F}(m_0)\rightarrow F\quad \mbox{for all}\ \ m_0\in\mathcal{H}\subset\mathcal{O}, 
\end{equation}
where $\mathcal{H}$ is a proper subset and it shall be described in more details in what follows. Here, we emphasise that in our study, we have no more restriction on $\Omega$ other than described above which makes our inverse problem study more realistic from a practical point of view. To be more specific, one can think of measuring/observing the space-time boundary data of $u$ and $m$ and from which to determine the running cost function $F$ over the space-time domain $Q:={\Omega}\times [0, T]$. We shall prove in a general setup that one can indeed achieve this goal. In fact, we shall establish sufficient conditions to guarantee that unique identifiability for \eqref{eq:ip1}, which is of primary importance in the theory of inverse problems and can be stated as follows: 
\begin{equation}\label{eq:ui1}
F_1=F_2\quad\mbox{if and only if}\quad \mathcal{N}_{F_1}(m_0)=\mathcal{N}_{F_2}(m_0)\ \ \mbox{for all}\ m_0\in\mathcal{H}, 
\end{equation}
where $F_j\in\mathcal{A}$, $j=1,2$, with $\mathcal{A}$ signifying an a-priori class that shall be detailed in what follows.

	\subsection{Technical developments and discussion }
	
As also discussed earlier, the forward problem of MFG systems has been extensively and intensively studied in the literature in recent years, and we refer to \cite{Car,high_order_der,Cira,FerGomTad} for more results on the well-posedness of the forward problems that are related to the MFG system \eqref{main} of the current study. Nevertheless, we would like to point out that for our inverse problem study, we need to establish new results for the forward MFG system \eqref{main}, especially on the regularity of the solutions and the high-order variation of the system around a fixed pair of solutions. On the other hand, the inverse problems for MFGs are much less studied in the literature. To our best knowledge, there are only several numerical results available in \cite{CFLNO,DOY} and in \cite{LMZ} unique identifiability results were established for an MFG inverse problem, which, though related, is of a different formulation from the one in the present article. Among several different aspects, a major technical difference between the studies in \cite{LMZ} and the current article is that the arguments in \cite{LMZ} critically depend on a high-order linearisation technique around a pair of trivial solutions $(u_0, m_0)\equiv (0, 0)$. This makes the study in \cite{LMZ} is unobjectionably unrealistic from a physical point of view, though mathematically rigorous. For the current study of \eqref{eq:ip1}--\eqref{eq:ui1}, due to the probability measure requirement on $m$, namely the constraint in \eqref{eq:distr1}. This technical constraint makes the corresponding inverse problem radically more challenging. In fact, in \cite{LMZ}, the high-order linearisation process around the trivial solutions can decouple the two PDEs of the MFG system in a certain sense. However, in the current one, we have to develop a new high-order linearisation process around a nontrivial pair of solutions, and the derived linearised systems are still coupled PDE systems. It is emphasised that even the treatments of the intermediate inverse problems for those linearised systems are technically novel to the literature. Furthermore, the constraint \eqref{eq:distr1} as well as the coupling of the MFG system significantly restrict the range of ``probing inputs", i.e. the initial distributions $m_0$, which in turn restrict the generation of ``many" probing modes $u$ and $m$ for detecting the unknown $F$. For general PDE inverse problems, it is generically held that the more dense of the set of PDE solutions (which play the role of the probing modes) can be generated, the more easily one can recover the unknowns. To overcome those challenges, we devise the inverse problem of recovering $F$ over $Q$ by measuring $u$ and $m$ on $\partial Q$ and in doing so, the free region $Q'\backslash Q$ enables us to construct a new and delicate family of the so-called CGO (Complex-Geometric-Optics) solutions to fulfil our inverse problem purpose. Moreover, our analysis in this aspect accompanies several novel and subtle estimates. We believe those technical developments shall be useful for tackling MFG inverse problems in different setups, and we chose to investigate these extensions in our future works. 

Finally, we would like to briefly mention that the mathematical study of inverse problems associated with nonlinear PDEs have received considerable interest in the literature recently; see e.g. \cite{KLU,LLLS,LLLS21,LLLZ} and the references cited therein. Even in such a context, there are several salient features of our study that are worth highlighting. First, the MFG system \eqref{main} couples the two nonlinear PDEs in a backward-forward manner with respect to the time. Second, there is a probability measure constraint \eqref{eq:distr1} which restricts the generation of ``probing modes". We construct a novel class of CGO solutions for the inverse problem study. Third, we develop a high-order linearisation method around a nontrivial pair of solutions. The resulting linearised systems are still coupled in nature. We believe the mathematical strategies developed in this article can find more applications in tackling other inverse problems associated with coupled nonlinear PDEs in different contexts. 

The rest of the paper is organized as follows. In Section 2, we fix some notations and introduce several auxiliary results as well as state the main result of the inverse problem. Section 3 is devoted to the study of the forward problem, and Sections~4 and 5 are devoted to the proof of the main result. 

%
%
%
%
%

	\section{Preliminaries and statement of main results}
	
	\subsection{Notations and Basic Setting}\label{notation}
	
For $k\in\mathbb{N}$ and $0<\alpha<1$, the H\"older space $C^{k+\alpha}(\overline{\Omega'})$ is defined as the subspace of $C^{k}(\overline{\Omega'})$ such that $\phi\in C^{k+\alpha}(\overline{\Omega'})$ if and only if $D^l\phi$ exist and are H\"older continuous with exponent $\alpha$ for all $l=(l_1,l_2,\ldots,l_n)\in \mathbb{N}^n$ with $|l|\leq k$, where $D^l:=\partial_{x_1}^{l_1}\partial_{x_2}^{l_2}\cdots\partial_{x_n}^{l_n}$ for $x=(x_1, x_2,\ldots, x_n)$. The norm is defined as
\begin{equation}
	\|\phi\|_{C^{k+\alpha}(\overline{\Omega'}) }:=\sum_{|l|\leq k}\|D^l\phi\|_{\infty}+\sum_{|l|=k}\sup_{x\neq y}\frac{|D^l\phi(x)-D^l\phi(y)|}{|x-y|^{\alpha}}.
\end{equation}
	If the function $\phi$ depends on both the time and space variables, we define $\phi\in C^{k+\alpha, \frac{k+\alpha}{2}}(Q')$ if $D^lD^{j}_t\phi$ exist and are are H\"older continuous with exponent $\alpha$ in $x$ and $\frac{k+\alpha}{2} $ in $t$ for all  $l\in \mathbb{N}^n$, $j\in\mathbb{N}$ with $|l|+2j\leq k.$ The norm is defined as
	\begin{equation}
		\begin{aligned}
			\|\phi\|_{ C^{k+\alpha, \frac{k+\alpha}{2}}(Q')}:&=\sum_{|l|+2j\leq k}\|D^lD^j_t\phi\|_{\infty}+\sum_{|l|+2j= k}\sup_{t,x\neq y}\frac{|\phi(x,t)-\phi(y,t)|}{|x-y|^{\alpha}}\\
			&+\sum_{|l|+2j= k}\sup_{t\neq t',x} \frac{|\phi(x,t)-\phi(x,t')|}{|t-t'|^{\alpha/2}}.
		\end{aligned}
	\end{equation}

	Moreover, we denote by $H^s(\Omega')$, $H^r(\Sigma')$, $H^s(0,T;H^r(\Omega'))$ the standard Sobolev spaces for $s,r\in\mathbb{R}.$
	
	Since we need to study the linearized system of $\eqref{main}$ for our inverse problem study, we next define the variation of a function defined on $\mathcal{P}(\Omega') $ (cf. \cite{num_boundary}). Recall that  $\mathcal{P}(\Omega')$ denotes the set of probability measures on $\Omega'$ and let $U$ be a real function defined on $\mathcal{P}(\Omega') $.
	  
	\begin{defi}\label{def_der_1}
		 Let $U :\mathcal{P}(\Omega')\to\mathbb{R}$. We say that $U$ is of class $C^1$ if there exists a continuous map $K:  \mathcal{P}(\Omega')\times \Omega'\to\mathbb{R}$ such that, for all $m_1,m_2\in\mathcal{P}(\Omega') $ we have
		 \begin{equation}\label{derivation}
		 \lim\limits_{s\to 0^+}\frac{U\big(m_1+s(m_2-m_1)-U(m_1)\big)}{s}=\int_{\Omega'} K(m_1,x)d(m_2-m_1)(x).
		 \end{equation}
	\end{defi}
    Note that the definition of $K$ is up to additive constants. We deinfe  the derivative
    $\dfrac{\delta U}{\delta m}$ as the unique map $K$ satisfying $\eqref{derivation}$ and 
    \begin{equation}
    	\int_{\Omega'} K(m,x) dm(x)=0.
    \end{equation}
Similarly, we can define higher order derivatives of $U$, and we refer to \cite{high_order_der} for more related discussion.
Finally, we  define the Wasserstein distance between $m_1$ and $m_2$ in $\mathcal{P}(\Omega')$, which shall be needed in studying the regularity of the derivative $\dfrac{\delta U}{\delta m}$.
\begin{defi}\label{W_distance}
	Let $m_1,m_2$ be two Borel probability measures on $\Omega'$. Define
	\begin{equation}
		d_1(m_1,m_2):=\sup_{Lip(\psi)\leq 1}\int_{\Omega'}\psi(x)d(m_1-m_2)(x),
	\end{equation}
	where $Lip(\psi)$ denotes the Lipschitz constant for a Lipschitz function, i.e., 
	\begin{equation}\label{eq:Lip1}
	Lip(\psi)=\sup_{x, y\in\Omega', x\neq y}\frac{|\psi(x)-\psi(y)|}{|x-y|}. 
	\end{equation}
\end{defi}
\begin{rmk}
	In Definitions~\ref{def_der_1} and \ref{W_distance}, $m$ ( i.e. $m_1$ or $m_2$ ) is viewed as a distribution. However, in other parts of the paper, we use $m$ to denote the density of a distribution such as in the MFG system \eqref{main}. 
\end{rmk}	
	\subsection{Well-posedness conditions and admissible class}\label{assump}
	
Throughout the rest of the paper and without loss of generality, we shall always assume that $|\Omega'|=1$ in order to simplify the exposition. Next, we introduce several assumptions which are generally needed in guaranteeing the well-posedness of the forward MFG system \eqref{main}; see \cite{num_boundary} and \cite{Cardaliaguet} for more related discussions on this aspect. 

	\begin{assum}\label{hypo}

			(I) (monotonicity property) For any $m_1,m_2\in\mathcal{P}(\Omega')$, we have 
		\begin{equation}\label{mono1}
			\int_{\Omega'} \big( F(x,m_1)-F(x,m_2)d(m_1-m_2)(x)\big)\geq 0,
		\end{equation}
		and 
		\begin{equation}
			\int_{\Omega'} \big( G(x,m_1)-G(x,m_2)d(m_1-m_2)(x)\big)\geq 0,
		\end{equation}
				for all $m_1,m_2\in \mathcal{P}(\Omega').$

		(II) (regularity conditions) Assume that
		\begin{equation}
			\sup_{m\in\mathcal{P}(\Omega')} \Big(\left\|F(\cdot,m)\right\|_{\alpha}+ \left\|\frac{\delta F}{\delta m}(\cdot,m,\cdot)\right\|_{\alpha,2+\alpha}\Big)+Lip (\frac{\delta F}{\delta m})< \infty,
		\end{equation}
		\begin{equation}
			\sup_{m\in\mathcal{P}(\Omega')} \Big(\left\|G(\cdot,m)\right\|_{2+\alpha}+ \left\|\frac{\delta G}{\delta m}(\cdot,m,\cdot)\right\|_{2+\alpha,2+\alpha}\Big)+ Lip (\frac{\delta G}{\delta m})< \infty,
		\end{equation}
		where
		\begin{equation}\label{eq:Lip2}
		\begin{split}
				Lip (\frac{\delta F}{\delta m}):=& \sup_{m_1\neq m_2}\Big (d_1(m_1-m_2)^{-1} \left\|\frac{\delta F}{\delta m_1}(\cdot,m,\cdot)- \frac{\delta F}{\delta m_1}(\cdot,m_2,\cdot)\right\|_{\alpha,1+\alpha}\Big)\\
		Lip (\frac{\delta G}{\delta m}):=& \sup_{m_1\neq m_2}\Big (d_1(m_1-m_2)^{-1} \left\|\frac{\delta G}{\delta m_1}(\cdot,m,\cdot)- \frac{\delta G}{\delta m_1}(\cdot,m_2,\cdot)\right\|_{2+\alpha,2+\alpha}\Big)
		\end{split}
		\end{equation}
			for all $m_1,m_2\in \mathcal{P}(\Omega').$ Here, we note the definitions of $Lip$ in \eqref{eq:Lip1} and \eqref{eq:Lip2} are slightly different, which should be clear from the context.

		(III)(compatibility conditions)
		\begin{equation}
			\begin{split}
			\left< D_y\frac{\delta G}{\delta m}(x,m,y),\nu(y)\right>&=0,\\
	        \left< D_y\frac{\delta F}{\delta m}(x,m,y),\nu(y)\right>&=0,\\
		   \left< D_x G(m,x),\nu(x)\right>&=0,
			\end{split}
		\end{equation}
		for all $m\in \mathcal{P}(\Omega').$
		\end{assum}
		
			Assumption~\ref{hypo} is mainly needed to guarantee the well-posedness of the forward MFG system \eqref{main}, and it shall be imposed throughout the rest of the paper.

Next, we introduce the admissible conditions on $F$ and $G$, which shall be mainly needed for our subsequent inverse problem study of \eqref{eq:ip1} and \eqref{eq:ui1}. 

\begin{defi}\label{Admissible class2}
		We say $U(x,z):\mathbb{R}^n\times\mathbb{C}\to\mathbb{C}$ is admissible, denoted by $U\in\mathcal{A}$, if it satisfies the following conditions:
		\begin{enumerate}
			\item[(i)] The map $z\mapsto U(\cdot,z)$ is holomorphic with value in $C^{\alpha}(\mathbb{R}^n)$ for some $\alpha\in(0,1)$.
			\item[(ii)] $U(x,1)=0$ for all $x\in\mathbb{R}^n$. Here we recall that we assume $|\Omega'|=1.$
			\item[(iii)] There exists $a_1\in\mathbb{R}$ such that $ U^{(1)}(x)>a_1>0$. 
		\end{enumerate} 	
	
		Clearly, if (1) and (2) are fulfilled, then $U$ can be expanded into a power series as follows:
		\begin{equation}\label{eq:G}
			U(x,z)=\sum_{k=1}^{\infty} U^{(k)}(x)\frac{(z-1)^k}{k!},
		\end{equation}
		where $ U^{(k)}(x)=\frac{\p^kU}{\p z^k}(x,1)\in C^{\alpha}(\mathbb{R}^n).$
	\end{defi}
	To make this definition clear, we have the following remark.
	\begin{rmk}
For the MFG system \eqref{main}, if we assume $F\in\mathcal{A}$, it means that $F$ possesses the power series expansion \eqref{eq:G} with the complex variable $z$ restricted on the real line. For this reason, we always assume in the series expansions \eqref{eq:G} that the coefficient functions $U^{(k)}$ are real-valued since $F$ is always real valued in our study. Moreover, we would like to emphasise that the admissibility conditions in Definition~\ref{Admissible class2} are compatible to those well-posedness conditions in Assumption~\ref{hypo}. In fact, there are overlaps among them and we shall not explore in details here since it is not the focus of our study. As a simple illustration, one can take $F(x, z)=e^z-1$ or $z-1$ and directly verify that all conditions can be fulfilled.  
	\end{rmk}
	
%


Next, we introduce the following function space, 
\begin{equation}\label{eq:fsh}
	H_{\pm}(Q):=\{u\in \mathcal{D}'(Q) \ \ |\ \ u\in L^2(Q) \text{ and } (\pm\p_t-\Delta)u\in L^2(Q) \},
\end{equation}
endowed with the norm
$$\|u\|^2_{H_{\pm}(Q)}:=\|u\|^2_{L^2(Q)}+\|(\pm\p_t-\Delta)u\|^2_{L^2(Q)}.$$ This function shall be needed in our analysis in Section~4 as well as the following admissibility condition for the terminal cost $G$ in \eqref{main}. 

\begin{defi}\label{improve_regular}
	We say $G$ is an action operator at $m=1$ if for any $\rho(x,t)\in H_{\pm}(Q)$, we have
	$$\dfrac{\delta G}{ \delta m}(x,1)(\rho(x,T)):=\left<\dfrac{\delta G}{ \delta m}(x,1,\cdot),\rho(x,T)\right>_{L^2}\in L^2(\Omega),$$
	and
	$$ \left\|  \dfrac{\delta G}{ \delta m}(x,1)(\rho(x,T)) \right\|_{L^2(\Omega)}\leq C \|\rho\|_{ L^2(Q) }.$$
	 We say $G\in\mathcal{B}$ if it is an action operator at $m=1$ and $ G(x,1)=B$ for some constant $B\in\mathbb{R}.$
\end{defi}

   Further discussions about $H_{\pm}$ shall be provided in section $\ref{sec_pre_est}$. In particular, we shall show that $\rho(x,T)$ is well-defined. 

	\begin{rmk}\label{rem:1}
	Similar to the admissibility condition on $F$, we would like to remark that the admissibility condition on $G$ in Definition~\ref{improve_regular} is also generic ones. In fact, in many practical setups, the terminal cost functions fulfil that the belong to the class $\mathcal{B}$; see \cite{high_order_der} for related examples. Here, we mention a typical one by letting
		$$
		G(x,m)=\int_{\Omega'}\psi(z,\rho*m(z))\rho(x-z)\ dz,
		$$
		where $*$ denotes the  convolutional operation and $\psi: \mathbb{R}^2\mapsto\mathbb{R}$ is a smooth map
		which is nondecreasing with respect to the second variable and $\rho$ is a smooth, even function with a
		compact support in $\Omega'$. Clearly, $G$ is an action operator at any point as long as it is in the form of convolution.
		\end{rmk}
		
	\begin{rmk}
		Note that if $F(x,1)=0$ and $G(x,1)=B$  (as in Definition $\ref{improve_regular}$), then
		 $(u,m)=(B,1)$ is a solution of the MFG system \eqref{main}. In this case, the initial distribution $m(x, 0)=1$. This is a common nature of MFG system that the uniform distribution is a stable state.
	\end{rmk}

\subsection{Main unique identifiability result}

	We are in a position to state the main result for the inverse problem \eqref{eq:ip1}-\eqref{eq:ui1}, which shows in a generic scenario that one can recover  the running cost $F$ from the measurement map $\mathcal{N}_F$. 
	
	\begin{thm}\label{der F}
		Let Assumption $\ref{hypo}$ holds for $ 0<\alpha<1$. Assume that $F_j \in\mathcal{A}$ ($j=1,2$) and $G\in\mathcal{B}$. Let $\mathcal{N}_{F_j}$ be the measurement map associated to
		the following system:
		\begin{equation}\label{eq:mfg1}
			\begin{cases}
				-\p_tu(x,t)-\Delta u(x,t)+\frac 1 2 {|\nabla u(x,t)|^2}= F_j(x,m(x,t)),& \text{ in }  Q',\medskip\\
				\p_t m(x,t)-\Delta m(x,t)-{\rm div} (m(x,t) \nabla u(x,t))=0,&\text{ in } Q',\medskip\\
				\p_{\nu} u(x,t)=\p_{\nu}m(x,t)=0      &\text{ on } \Sigma',\medskip\\
				u(x,T)=G(x,m(x,T)), & \text{ in } \Omega',\medskip\\
				m(x,0)=m_0(x), & \text{ in } \Omega'.\\
			\end{cases}  		
		\end{equation}	
		If for any $m_0\in  C^{2+\alpha}(\Omega') \cap \mathcal{O}$, where $\mathcal{O}$ is defined in \eqref{eq:distr1}, one has 
		$$\mathcal{N}_{F_1}(m_0)=\mathcal{N}_{F_2}(m_0),$$    then it holds that 
		$$F_1(x,z)=F_2(x,z)\ \text{  in  }\ \Omega\times \mathbb{R}.$$ 
	\end{thm}

	
	\section{Auxiliary results on the forward problem}\label{section wp}
	
	In this section, we derive several auxiliary results on the forward problem of the MFG system \eqref{main}. One of the key results is the infinite differentiability of the system with respect to small variations around (the density of) a uniform
	distribution ($m_0(x)$). First, we show the existence result of an auxiliary system. 
	
	\begin{lem}\label{linear app unique}
		Assume that   $F^{(1)}\in C^{\alpha}(\Omega')$.
	For  any  $g,\tilde g\in C^{2+\alpha}(\overline\Omega')$,  
	and  $h,\tilde{h}\in C^{\alpha,\alpha/2}(\overline Q')$ with  the   compatibility conditions:
	\begin{align}\label{c-systems}
		\p_{\nu}\tilde g(x)= 
		\p_{\nu} g(x)=0,
	\end{align}
 the following system
		\begin{equation}\label{surjective}
			\begin{cases}
				-u_t-\Delta u-F^{(1)}(x)m=h & \text{ in } Q',\medskip\\
				m_t-\Delta m-\Delta u=\tilde{h}  & \text{ in } Q',\medskip\\
				\p_{\nu}u(x,t)=\p_{\nu}m(x,t)=0     & \text{ on } \Sigma',\medskip\\
				\displaystyle{u(x,T)=\frac{\delta G}{\delta m}(x,1)(m(x,T))+g}     & \text{ in } \Omega',\medskip\\
				m(x,0)=\tilde{g} & \text{ in } \Omega'.\\
			\end{cases}
		\end{equation}
	admits a pair of solutions $(u,m)\in [ C^{2+\alpha,1+\alpha/2}(\overline Q')]^2$. 
	\end{lem}
\begin{proof}
   This is consequence of Proposition 5.8 in \cite{num_boundary}. In fact, by Proposition 5.8 in \cite{num_boundary}, we have that \eqref{surjective} admits a pair of solutions $(u,m)$ such that $u\in C^{2+\alpha,1+\alpha/2}(\overline Q')$. Since $m(x,0)=\tilde{g} \in C^{2+\alpha}(\overline\Omega')$ and $F\in C^{\alpha}(\Omega')$ in this case, we have
   $m\in C^{2+\alpha,1+\alpha/2}(\overline Q') $.
\end{proof}

Now we present the proof of the local well-posedness of the MFG system \eqref{main}, which shall be needed in our subsequent inverse problem study.  

	
	\begin{thm}\label{local_wellpose}
	Let Assumption $\ref{hypo}$ hold for $ 0<\alpha<1$. Suppose that $F\in\mathcal{A}$ and $G\in\mathcal{B}$. The following results hold:
		\begin{enumerate}
			
			\item[(a)]
			There exist constants $\delta>0$ and $C>0$ such that for any 
			\[
			m_0\in B_{\delta}(C^{2+\alpha}(\Omega') :=\{m_0\in C^{2+\alpha}(\Omega'): \|m_0\|_{C^{2+\alpha}(\Omega')}\leq\delta \},
			\]
			the MFG system $\eqref{main}$ has a solution $(u,m)\in
			[C^{2+\alpha,1+\frac{\alpha}{2}}(Q)]^2$ which satisfies
			\begin{equation}\label{eq:nn1}
				\|(u,m)\|_{ C^{2+\alpha,1+\frac{\alpha}{2}}(Q')}:= \|u\|_{C^{2+\alpha,1+\frac{\alpha}{2}}(Q')}+ \|m\|_{C^{2+\alpha,1+\frac{\alpha}{2}}(Q')}\leq C\|m_0\|_{ C^{2+\alpha}(\Omega')}.
			\end{equation}
			Furthermore, the solution $(u,m)$ is unique within the class
			\begin{equation}\label{eq:nn2}
				\{ (u,m)\in  C^{2+\alpha,1+\frac{\alpha}{2}}(Q')\times C^{2+\alpha,1+\frac{\alpha}{2}}(Q'): \|(u,m)\|_{ C^{2+\alpha,1+\frac{\alpha}{2}}(Q')}\leq C\delta \}.
			\end{equation}
			
			\item[(b)] Define a function 
			\[
			S: B_{\delta}(C^{2+\alpha}(\Omega')\to C^{2+\alpha,1+\frac{\alpha}{2}}(Q')\times C^{2+\alpha,1+\frac{\alpha}{2}}(Q')\ \mbox{by $S(m_0):=(u,v)$}, 
			\] 
			where $(u,v)$ is the unique solution to the MFG system \eqref{main}.
			Then for any $m_0\in B_{\delta}(C^{2+\alpha}(\Omega'))$, $S$ is holomorphic.
		\end{enumerate}
	\end{thm}
	\begin{proof}
		Let 
		\begin{align*}
			&X_1:= \{ m\in C^{2+\alpha}(\Omega' ): \p_{\nu}m=0  \} , \\
			&X_2:=\{  (u,m)\in  C^{2+\alpha,1+\frac{\alpha}{2}}(Q')\times C^{2+\alpha,1+\frac{\alpha}{2}}(Q') : \p_{\nu}m=\p_{\nu}u=0 \text{ on } \Sigma' \}  \},\\
			&X_3:=X_1\times X_1\times C^{\alpha,\frac{\alpha}{2}}(Q' )\times C^{\alpha,\frac{\alpha}{2}}(Q'),
		\end{align*} and we define a map $\mathscr{K}:X_1\times X_2 \to X_3$ by that for any $(m_0,\tilde u,\tilde m)\in X_1\times X_2$,
		\begin{align*}
			&
			\mathscr{K}( m_0,\tilde u,\tilde m)(x,t)\\
			:=&\big( \tilde u(x,T)-G(x,\tilde m(x,T)), \tilde m(x,0)-m_0(x) , 
			-\p_t\tilde u(x,t)-\Delta \tilde u(x,t)\\ &+\frac{|\nabla \tilde u(x,t)|^2}{2}- F(x,t,\tilde m(x,t)), 
			\p_t \tilde m(x,t)-\Delta \tilde m(x,t)-{\rm div}(\tilde m(x,t)\nabla \tilde u(x,t))  \big) .
		\end{align*}

		First, we show that $\mathscr{K} $ is well-defined. Since the
		H\"older space is an algebra under the point-wise multiplication, we have $|\nabla u|^2, {\rm div}(m(x,t)\nabla u(x,t))  \in C^{\alpha,\frac{\alpha}{2}}(Q' ).$
		By the Cauchy integral formula,
		\begin{equation}\label{eq:F1}
			F^{(k)}\leq \frac{k!}{R^k}\sup_{|z|=R}\|F(\cdot,\cdot,z)\|_{C^{\alpha,\frac{\alpha}{2}}(Q' ) },\ \ R>0.
		\end{equation}
		Then there is $L>0$ such that for all $k\in\mathbb{N}$,
		\begin{equation}\label{eq:F2}
			\left\|\frac{F^{(k)}}{k!}m^k\right\|_{C^{\alpha,\frac{\alpha}{2}}(Q' )}\leq \frac{L^k}{R^k}\|m\|^k_{C^{\alpha,\frac{\alpha}{2}}(Q' )}\sup_{|z|=R}\|F(\cdot,\cdot,z)\|_{C^{\alpha,\frac{\alpha}{2}}(Q' ) }.
		\end{equation}
		By choosing $R\in\mathbb{R}_+$ large enough and by virtue of \eqref{eq:F1} and \eqref{eq:F2}, it can be seen that the series  converges in $C^{\alpha,\frac{\alpha}{2}}(Q' )$ and therefore $F(x,m(x,t))\in  C^{\alpha,\frac{\alpha}{2}}(Q' ).$  
		
		Using the compatibility assumption and regularity conditions on $G$, we have that $\mathscr{K} $ is well-defined.

		Let us show that $\mathscr{K}$ is holomorphic. Since $\mathscr{K}$ is clearly locally bounded, it suffices to verify that it is weakly holomorphic. That is we aim to show that the map
		$$\lambda\in\mathbb C \mapsto \mathscr{K}((m_0,\tilde u,\tilde m)+\lambda (\bar m_0,\bar u,\bar m))\in X_3,\quad\text{for any $(\bar m_0,\bar u,\bar m)\in X_1\times X_2$}$$
		is holomorphic. In fact, this follows from the condition that $F\in\mathcal{A}$ and $G\in\mathcal{B}$.

		Note that $  \mathscr{K}(\frac{1}{|\Omega'|},0,\frac{1}{|\Omega'|})=\mathscr{K}(1,0,1)= 0$. Let us compute $\nabla_{(\tilde u,\tilde m)} \mathscr{K} (1,0,1)$:
		\begin{equation}\label{Fer diff}
			\begin{aligned}
				\nabla_{(\tilde u,\tilde m)} \mathscr{K}(1,0,1) (u,m)& =( u|_{t=T}-\frac{\delta G}{\delta m}(x,1)(m(x,T)), m|_{t=0}, \\
				&-\p_tu(x,t)-\Delta u(x,t)-F^{(1)}m, \p_t m(x,t)-\Delta m(x,t)-\Delta u).
			\end{aligned}			
		\end{equation}
		
         On the one hand, $\nabla_{(\tilde u,\tilde m)} \mathscr{K} (1,0,1)$  is injective. In fact, if $\nabla_{(\tilde u,\tilde m)} \mathscr{K} (1,0,1)(u,m)=0$, then 	we have 
       	\begin{equation}
       	\begin{cases}
       		-u_t-\Delta u= F^{(1)}(x)m & \text{ in } Q',\medskip\\
       		m_t-\Delta m-\Delta u=0  & \text{ in } Q',\medskip\\
       		\p_{\nu}u(x,t)=\p_{\nu}m(x,t)=0     & \text{ on } \Sigma',\medskip\\
       		\displaystyle{u(x,T)=\frac{\delta G}{\delta m}(x,1)(m(x,T))  }   & \text{ in } \Omega',\medskip\\
       		m(x,0)=0 & \text{ in } \Omega'.\\
       	\end{cases}
       \end{equation}
   Notice that
    \begin{equation}\label{eq:dd1}
    \begin{split}
   	&\int_{\Omega}\quad u_tm+um_t\ dx\\
   	=&\int_{\Omega} (-\Delta u- F^{(1)}m)m + (\Delta m+\Delta u)u\ dx\\
   	=&\int_{\Omega} -F^{(1)}m^2-|\nabla u|^2\ dx\\
   	\leq& \int_{\Omega} -F^{(1)}m^2\ dx. 
   \end{split}
   \end{equation}
   Integrating both sides of \eqref{eq:dd1} from $0$ to $T$, we can obtain
   \begin{equation}
   	\int_{\Omega}\frac{\delta G}{\delta m}(x,1)(m(x,T)) m(x,T) dx\leq -\int_{Q} F^{(1)}m^2 dx.
   \end{equation}
Since  $F^{(1)}(x)$ is positive and $G$ satisfies the monotonicity property, we readily see that $m=0$ and then $u=0$. Hence, $\nabla_{(\tilde u,\tilde m)} \mathscr{K} (1,0,1)$  is injective.

On the other hand, by Lemma $\ref{linear app unique}$, $\nabla_{(\tilde u,\tilde m)} \mathscr{K} (1,0,1)$ is surjective. Therefore, $\nabla_{(\tilde u,\tilde m)} \mathscr{K} (1,0,1)$ is a linear isomorphism between $X_2$ and $X_3$. Hence, by the Implicit Function Theorem, there exist $\delta>0$ and a unique holomorphic function $S: B_{\delta}(\Omega')\to X_2$ such that $\mathscr{K}(m_0,S(m_0))=0$ for all $m_0\in B_{\delta}(\Omega') $.
		
		By letting $(u,m)=S(m_0)$, we obtain the unique solution of the MFG system \eqref{main}. Let $ (u_0,v_0)=S(0)$.   Since $S$ is Lipschitz, we know that there exist constants $C,C'>0$ such that 
		\begin{equation*}
			\begin{aligned}
				&\|(u,m)\|_{ C^{2+\alpha,1+\frac{\alpha}{2}}(Q')^2}\\
				\leq& C'\|m_0\|_{B_{\delta}(\Omega')} +\|u_0\|_	{ C^{2+\alpha,1+\frac{\alpha}{2}}(Q')}+\|v_0\|_{ C^{2+\alpha,1+\frac{\alpha}{2}}(Q')}\\
				\leq& C \|m_0\|_{B_{\delta}(\Omega')}.
			\end{aligned}
		\end{equation*}

		The proof is complete. 
	\end{proof}
\begin{rmk}
	In the proof of this theorem, we do not make use of the fact that the initial data $m_0$ is a density of a distribution. It is not necessary in the proof of the well-posedness of the forward problem. On the other hand, it is noted that if we choose $m_0$ to be a density of a distribution, the KFP equation forces the solution $m$ to be a density of a distribution for all $t\in (0,T).$ However, when we consider the inverse problem, we need to restrict our discussion to the case that $m_0$ is a density.
\end{rmk}

	\section{ A-priori estimates and analysis of the linearized systems  }\label{analysis of lin}
	
		\subsection{Higher-order linearization}\label{HLM}
		
	We next develop a high-order linearization scheme of the MFG system \eqref{main} in the probability space around a uniform distribution. This method depends on the infinite differentiability of the system with respect to a given input $m_0(x)$, which was established in Theorem~$\ref{local_wellpose}$.

	First, we introduce the basic setting of this higher order
	linearization method. Consider the system $\eqref{main}$. Let 
	$$m_0(x;\varepsilon)=\frac{1}{|\Omega'|}+\sum_{l=1}^{N}\varepsilon_lf_l=1+ \sum_{l=1}^{N}\varepsilon_lf_l,$$
	where 
	\[
	f_l\in C^{2+\alpha}(\mathbb{R}^n)\quad\mbox{and}\quad\int_{\Omega'} f_l(x) dx =0,
	\]
	 and $\varepsilon=(\varepsilon_1,\varepsilon_2,...,\varepsilon_N)\in\mathbb{R}^N$ with 
	$|\varepsilon|=\sum_{l=1}^{N}|\varepsilon_l|$ small enough. 
	Clearly, $m_0(x;\varepsilon)$ is a density of a distribution in $\mathcal{P}(\Omega').$
	By Theorem $\ref{local_wellpose}$, there exists a unique solution $(u(x,t;\varepsilon),m(x,t;\varepsilon) )$ of $\eqref{main}$. If $\varepsilon=0,$ by our assumption, we have $(u(x,t;0),m(x,t;0) )= (B,1)$ for some $B\in\mathbb{R}.$
	
	Let
	$$u^{(1)}:=\p_{\varepsilon_1}u|_{\varepsilon=0}=\lim\limits_{\varepsilon\to 0}\frac{u(x,t;\varepsilon)-u(x,t;0) }{\varepsilon_1},$$
	$$m^{(1)}:=\p_{\varepsilon_1}m|_{\varepsilon=0}=\lim\limits_{\varepsilon\to 0}\frac{m(x,t;\varepsilon)-m(x,t;0) }{\varepsilon_1}.$$
	
	The idea is that we consider a new system of $(u^{(1)},m^{(1)}).$
	It is sufficient for us to only consider this system in $Q:=\Omega\times[0,T]$. Now, we have  that $(u_{j}^{(1)},m_{j}^{(1)} )$ satisfies the following system:
	\begin{equation}\label{linear l=1,eg}
		\begin{cases}
			-\p_tu^{(1)}(x,t)-\Delta u^{(1)}(x,t)= F^{(1)}(x)m^{(1)}(x,t)& \text{ in } Q,\medskip\\
			\p_t m^{(1)}(x,t)-\Delta m^{(1)}(x,t)-\Delta  u^{(1)}(x,t)=0&\text{ in } Q,\medskip\\
			\displaystyle{u^{(1)}_j(x,T)=\frac{\delta G}{\delta m}(x,1)(m^{(1)}(x,T))} & \text{ in } \Omega,\medskip\\
			m^{(1)}_j(x,0)=f_1(x). & \text{ in } \Omega.\\
		\end{cases}  
	\end{equation}
Since $\Omega$ is a closed proper subset of $\Omega'$, in this system, $f_1(x)$ can be  arbitrary $C^{2+\alpha}$ function in $\Omega.$
    
	Then we can define $$u^{(l)}:=\p_{\varepsilon_l}u|_{\varepsilon=0}=\lim\limits_{\varepsilon\to 0}\frac{u(x,t;\varepsilon)-u(x,t;0) }{\varepsilon_l},$$
	$$m^{(l)}:=\p_{\varepsilon_l}m|_{\varepsilon=0}=\lim\limits_{\varepsilon\to 0}\frac{m(x,t;\varepsilon)-m(x,t;0) }{\varepsilon_l},$$
	for all $l\in\mathbb{N}$ and obtain a sequence of similar systems.
	In the proof of Theorem $\ref{der F}$, we recover the first Taylor coefficient of $F$  by considering this new system $\eqref{linear l=1,eg}$. In order to recover the higher order Taylor coefficients, 
	we consider 
	\begin{equation}\label{eq:ht1}
		u^{(1,2)}:=\p_{\varepsilon_1}\p_{\varepsilon_2}u|_{\varepsilon=0},
		m^{(1,2)}:=\p_{\varepsilon_1}\p_{\varepsilon_2}m|_{\varepsilon=0}.
	\end{equation}
	
 We have the second-order linearization as follows:
	\begin{equation}\label{linear l=1,2 eg}
		\begin{cases}
			-\p_tu^{(1,2)}-\Delta u^{(1,2)}(x,t)+\nabla u^{(1)}\cdot \nabla u^{(2)}\medskip\\
			\hspace*{3cm}= F^{(1)}(x)m^{(1,2)}+F^{(2)}(x)m^{(1)}m^{(2)},& \text{ in } \Omega\times(0,T),\medskip\\
			\p_t m^{(1,2)}-\Delta m^{(1,2)}-\Delta u^{(1,2)}= {\rm div} (m^{(1)}\nabla u^{(2)})+{\rm div}(m^{(2)}\nabla u^{(1)}) ,&\text{ in } \Omega\times (0,T),\medskip\\
			\displaystyle{u^{(1,2)}(x,T)=\frac{\delta G}{\delta m}(x,1)m^{(1,2)}(x,T)+\frac{\delta^2G}{\delta m^2}(x,1)(m^{(1)}(x,T)m^{(2)}(x,T)),} & \text{ in } \Omega,\medskip\\
			m^{(1,2)}(x,0)=0, & \text{ in } \Omega.\\
		\end{cases}  	
	\end{equation}
	Notice that the non-linear terms of the system $\eqref{linear l=1,2 eg}$ depend on the first-order linearised system $\eqref{linear l=1,eg}$. Since we shall make use of the mathematical induction to recover the high-order Taylor coefficients of $F$. This shall be an important ingredient in our proof of Theorem~\ref{der F} in what follows. 
	
Similarly, for $N\in\mathbb{N}$, we consider 
	\begin{equation*}
		u^{(1,2...,N)}=\p_{\varepsilon_1}\p_{\varepsilon_2}...\p_{\varepsilon_N}u|_{\varepsilon=0},
	\end{equation*}
	\begin{equation*}
		m^{(1,2...,N)}=\p_{\varepsilon_1}\p_{\varepsilon_2}...\p_{\varepsilon_N}m|_{\varepsilon=0}.
	\end{equation*}
	we can obtain a sequence of parabolic systems, which shall be employed again in determining the higher order Taylor coefficients of the unknowns $F$ . 
	
	\subsection{A-priori estimates}\label{sec_pre_est}
	Let us consider the linearised system:
			\begin{equation}
		\begin{cases}
			-u_t-\Delta u= F^{(1)}(x)m & \text{ in } Q,\medskip\\
			m_t-\Delta m-\Delta u=0  & \text{ in } Q,\medskip\\
			\displaystyle{u(x,T)=\frac{\delta G}{\delta m}(x,1)(m(x,T))}    & \text{ in } \Omega,\medskip\\
			m(x,0)=f(x) & \text{ in } \Omega.\\
		\end{cases}
	\end{equation}
      We need derive several quantitative a-priori estimates of this linearised system. Since we also need to focus on the duality of this system, we claim similar results for the backward parabolic equation. We first present two auxiliary Carleman estimates. 
      
		\begin{lem}\label{Carleman 1}
		Let $\lambda>0$ and  $v:=e^{\lambda^2t}u(x,t)\in L^2(0,T;H^2(\Omega))\cap H^1(0,T; L^2(\Omega))$  ,$u(x,T)=0$ and $u(x,t)=0$ in $\Sigma$. Then 
		\begin{equation}\label{est1}
			\int_Q e^{2\lambda^2 t}(-\p_t u-\Delta u)^2 dxdt \geq C \Big[\lambda^4\int_Q    v^2 dxdt+ \int_Q (\Delta v)^2dxdt+\lambda^2\int_Q |\p_{\nu}v|^2 dxdt\Big].
		\end{equation}
	\end{lem}
	\begin{proof}
		Notice that
		$$ e^{2\lambda^2 t}(-\p_t u-\Delta u)^2=( -\p_t v-\Delta v + \lambda^2 v )^2.$$
		Therefore,
		\begin{align*}
			&\int_Q e^{2\lambda^2 t}(-\p_t u-\Delta u)^2 dxdt\\
			=   &\int_Q ( -\p_t v-\Delta v + \lambda^2 v )^2 dxdt\\
			\geq &\lambda^4\int_Q  v^2  + (\Delta v)^2  dxdt + 2\int_Q \p_tv \Delta v dxdt- 2\lambda^2\int_Q (\Delta v )v+(\p_tv )v dxdt.\\
		\end{align*}
		Note that
		\begin{align*}
			&2\int_Q \p_tv \Delta v dxdt =-\int_Q \p_t  |\nabla v|^2 dxdt=\int_{\Omega} |\nabla v(x,0)|^2 dx,\\
			&-2\int_Q (\Delta v )v dxdt= 2\int_Q |\nabla v|^2 dxdt\geq 2C\int_Q |\p_{\nu}v|^2 dxdt,\\
			&2\int_Q (\p_tv )v dxdt=\int_Q \p_t v^2 dxdt= -\int_{\Omega} v(x,0)^2 dx.\\
		\end{align*}
		It follows that $\eqref{est1}$ is true.
		
		The proof is complete. 
	\end{proof}
	
Similarly, we have
	\begin{lem}\label{Carleman 2}
		Let $\lambda>0$ and $v:=e^{-\lambda^2t}u(x,t)\in L^2(0,T;H^2(\Omega))\cap H^1(0,T; L^2(\Omega))$. Then it holds that
		\begin{equation}\label{est2}
			\int_Q e^{-2\lambda^2 t}(\p_t u-\Delta u)^2 dxdt \geq C \Big[\lambda^4\int_Q    v^2 dxdt+ \int_Q (\Delta v)^2dxdt+\lambda^2\int_Q |\p_{\nu}v|^2 dxdt\Big].
		\end{equation}
	\end{lem}
	
	Next, we recall the function space $H_\pm(Q)$ defined in \eqref{eq:fsh}. 
	We  use the following facts about these spaces. First, $C^{\infty}(Q)$ is dense $H_{\pm}.$ Second, we can define the trace operator on $H_{\pm}$. In particular, we can define $u(x,0),u(x,T)\in H^{-1}(\Omega).$ Hence, we have Green's formula in $H_{\pm}$; see also \cite{cgo} for related study.
	
\begin{lem}\label{prior estimate}
	Suppose there are $M>\varepsilon>0$ such that $\varepsilon<F(x)<M$ and $G\in \mathcal{B}$.  Let $(u,m)\in \big(L^2(0,T; H_0^1(\Omega))\cap H^1(0,T;H^{-1}(\Omega)\big)\times H_{+}(Q)$ be a solution to the following system
	\begin{equation}\label{prior estimate 1 }
		\begin{cases}
			-u_t-\Delta u= F(x)m & \text{ in }\ Q,\medskip\\
			m_t-\Delta m-\Delta u=0  & \text{ in }\ Q,\medskip\\
			u(x,t)=m(x,t)=0     & \text{ on }\ \Sigma,\medskip\\
			\displaystyle{u(x,T)= \frac{\delta G}{\delta m}(x,1)(m(x,T) )}   & \text{ in }\ \Omega.\\
		\end{cases}
	\end{equation}
	we have 
	$$ \|m\|_{L^2(Q)}\leq C\|m(x,0)\|_{H^{-1}(\Omega)}.$$
\end{lem}
\begin{proof}
	Notice that
	\begin{equation}\label{eq:ddd2}
	\begin{split}
		&\int_{\Omega}\quad u_tm+um_t\quad dx\\
		=&\int_{\Omega} (-\Delta u- Fm)m + (\Delta m+\Delta u)u dx\\
		=&\int_{\Omega} -Fm^2-|\nabla u|^2 dx\\
		\leq& \int_{\Omega} -Fm^2 dx. 
	\end{split}
	\end{equation}
	Integrating both sides of \eqref{eq:ddd2} from $0$ to $T$, we can obtain that
	\begin{equation}
		\int_{\Omega} \frac{\delta G}{\delta m}(x,1)(m(x,T))m(x,T)\ dx-\int_{\Omega} u(x,0)m(x,0)\ dx\leq -\int_{Q} Fm^2\ dx.
	\end{equation}
	Since $G\in\mathcal{B}$, the standard a-prior estimate for the parabolic equation (cf. \cite{evans}) yields that
	\begin{equation}\label{control norm of H_+}
	\|u(x,0)\|_{H_0^1(\Omega)}\leq C(\|m\|_{L^2(Q)}+   \|\frac{\delta G}{\delta m}(x,1)(m(x,T) )\|_{L^2(Q)}) .
	\end{equation}
	 Then by the assumption $G\in\mathcal{B}$, we have 
	\begin{align*}
		\varepsilon \|m\|^2_{L^2(Q)}&\leq \int_{\Omega} Fm^2 dx\\
		&\leq \int_{\Omega} u(x,0)m(x,0) dx\\
		&\leq C    (\|m\|_{L^2(Q)}+   \|\frac{\delta G}{\delta m}(x,1)(m(x,T) )\|_{L^2(Q)})     \|m(x,0)\|_{H^{-1}(\Omega)}\\
		&\leq C\|m\|_{L^2(Q)} \|m(x,0)\|_{H^{-1}(\Omega)},
	\end{align*}
which readily proves the statement of the lemma and competes its proof. 
\end{proof}	

	By following a similar argument, one can show that 
	
\begin{lem}\label{prior estimate2}
	Suppose there is $M>\varepsilon>0$ such that $\varepsilon<F(x)<M$.  Let $(u,m)\in \big(L^2(0,T; H_0^1(\Omega))\cap H^1(0,T;H^{-1}(\Omega)\big)\times H_{-}(Q)$ be a solution to the following system
	\begin{equation}\label{prior estimate 2}
		\begin{cases}
			u_t-\Delta u= F(x)m & \text{ in } Q,\medskip\\
			-m_t-\Delta m-\Delta u=0  & \text{ in } Q,\medskip\\
			u(x,t)=m(x,t)=0     & \text{ on } \Sigma,\medskip\\
			u(x,0)= 0 & \text{ in } \Omega.\\
		\end{cases}
	\end{equation}
	we have 
	$$ \|m\|_{L^2(Q)}\leq C\|m(x,T)\|_{H^{-1}(\Omega)}.$$
\end{lem}

	\subsection{Construction of CGO solutions}
	
	In this subsection, we focus on constructing certain solutions of a special form for the linearized systems, which shall serve as the ``probing modes" for our inverse problem. Those special solutions are referred to as the CGO (complex geometric optic) solutions. 
	
Next, we consider the weighted Hilbert space $L^2(Q; e^{2\lambda^2t})$ and $L^2(Q; e^{-2\lambda^2t}) $ with the scalar products 
	$$\left<u,v\right>_{\lambda^2}=\int_Q\quad u(x,t)v(x,t) e^{2\lambda^2t } \quad dxdt,$$
	$$\left<u,v\right>_{-\lambda^2}=\int_Q\quad u(x,t)v(x,t) e^{-2\lambda^2t } \quad dxdt,$$
	respectively. Similarly, we define $L^2(\Sigma; e^{2\lambda^2t})$ and $L^2(\Sigma; e^{-2\lambda^2t})$.

\begin{thm}\label{construct CGO 1}
	Consider the system
	\begin{equation}\label{CGO}
		\begin{cases}
			-u_t-\Delta u= F(x)m & \text{ in } Q,\medskip\\
			m_t-\Delta m-\Delta u=0  & \text{ in } Q,\medskip\\
			m(x,t)=u(x,t)=0  & \text{ on } \Sigma,\medskip\\
			\displaystyle{u(x,T)=\frac{\delta G}{\delta m}(x,1)(m(x,T)) }    & \text{ in } \Omega. \\
		\end{cases}
	\end{equation}
	Assume that $F$ is uniformly bounded by $M>0$ and $G\in\mathcal{B}$.  Let $\theta_{+}=1-e^{-\lambda^{3/4}t}$ where $\lambda$ is large enough (depending only on $M$ and $\Omega$). Then we have that:
	\begin{enumerate}
	\item[(i)] the system $\eqref{CGO}$ has a solution in the form 
	$$(u,m)=\Big(u,  e^{-\lambda^2t-\lambda\mathrm{i}x\cdot\xi}(\theta_{+}  e^{-\mathrm{i}(x,t)\cdot (\eta,\tau)}+ w_{+}(x,t) )\Big),$$
	such that $\xi,\eta\in\mathbb{S}^{n-1}$, $\xi\cdot\eta=0$, $\tau\in\mathbb{R}$ and 
	\begin{equation}
		\begin{aligned}
			&\lim\limits_{\lambda\to\infty} \|w_{+}\|_{L^2(Q)}=0.
		\end{aligned}
	\end{equation}
	Here and also in what follows, $\mathrm{i}:=\sqrt{-1}$ is the imaginary unit. 
	
	\item[(ii)] the solution $(u,m)$ in (i) satisfies $u\in L^2(0,T; H_0^1(\Omega))\cap H^1(0,T;H^{-1}(\Omega))$ and $m(x,t)\in H_{+}.$
	\end{enumerate}
\end{thm}

\begin{proof}

	(i).~Define  $\rho=e^{-\lambda^2t- \lambda\mathrm{i}x\cdot\xi} w_{+}:=\psi w_{+}.$
	It can be directly verified that $(u,\rho)$ is a solution of the following system:
	\begin{equation}\label{CGO'}
		\begin{cases}
			-u_t-\Delta u-F\rho = F(x)\theta_{+}\psi e^{-\mathrm{i}(x,t)\cdot (\eta,\tau)} & \text{ in } Q,\medskip\\
			\rho_t-\Delta\rho-\Delta u=-\psi e^{-\mathrm{i}(x,t)\cdot (\eta,\tau)}[\lambda^{3/4}e^{-\lambda^{3/4}t}-\mathrm{i}\theta_{+}\tau-\theta_{+}]  & \text{ in } Q,\medskip\\
			u(x,t)=0 & \text{ on } \Sigma,\medskip\\
			\rho(x,t)=e^{-\lambda^2t-\lambda\mathrm{i}x\cdot\xi}(\theta_{+}  e^{-\mathrm{i}(x,t)\cdot (\eta,\tau)} ) & \text{ on } \Sigma,\medskip\\
			\displaystyle{u(x,T)= \frac{\delta G}{\delta m}(x,1)(m(x,T))}        & \text{ in } \Omega. \\
		\end{cases}
	\end{equation}

	Next, we define a map $\mathcal{M}$ form $ J:= \{\rho : e^{\lambda^2t}\rho\in H_{+}(\Omega) \} $ to itself as follows. 	Let $\rho\in J\subset L^2(Q; e^{2\lambda^2t}) $ and $u$ be the solution of the system:
	\begin{equation}\label{define M pre}
		\begin{cases}
			-u_t-\Delta u-F\rho = F(x)\theta_{+}\psi e^{-\mathrm{i}(x,t)\cdot (\eta,\tau)} & \text{ in } Q,\medskip\\
			u(x,t)=0   & \text{ on } \Sigma, \medskip\\
			\displaystyle{u(x,T)=\frac{\delta G}{\delta m}(x,1)\Big(e^{-\lambda^2T-\lambda\mathrm{i}x\cdot\xi}\theta_{+}(T)  e^{-\mathrm{i}(x,T)\cdot (\eta,\tau)}+ \rho(x,T)\Big)} & \text{ in } \Omega.\\
		\end{cases}
	\end{equation}	
	Then we consider the following system
	\begin{equation}\label{define M}
		\begin{cases}
			(\mathcal{M}(\rho))_t-\Delta(\mathcal{M}(\rho))-\Delta u=-\psi e^{-\mathrm{i}(x,t)\cdot (\eta,\tau)}[\lambda^{3/4}e^{-\lambda^{3/4}t}-\mathrm{i}\theta_{+}\tau-\theta_{+}]  & \text{ in } Q, \medskip\\ 
			(\mathcal{M}(\rho))(x,t)= e^{-\lambda^2t-\lambda\mathrm{i}x\cdot\xi}(\theta_{+}  e^{-\mathrm{i}(x,t)\cdot (\eta,\tau)} )& \text{ on } \Sigma. \\
		\end{cases}
	\end{equation}
	Let 
	\[
	H(x,t)= -\psi e^{-\mathrm{i}(x,t)\cdot (\eta,\tau)}[\lambda^{3/4}e^{-\lambda^{3/4}t}-\mathrm{i}\theta_{+}\tau-\theta_{+}],
	\]
	and 
	\[
	h(x,t)=e^{-\lambda^2t-\lambda\mathrm{i}x\cdot\xi}(\theta_{+}  e^{-\mathrm{i}(x,t)\cdot (\eta,\tau)} ).
	\]
	Next, we define a map $\mathcal{L}:X  \to \mathbb{R}$, where $$ X:=\{( -\p_t-\Delta)f : f(x,t)\in C_0^{2}(\overline{Q}) \},$$
	by 
	\[
	\mathcal{L}(( -\p_t-\Delta)f)=\left< f,\Delta u+H(x,t) \right>_0-\left<\p_{\nu}f,h(x,t)\right>_0.
	\]

	By Lemma $\ref{Carleman 1}$, we have
	\begin{equation}\label{bouned map}
		\begin{aligned}
			&| \mathcal{L}(( -\p_t-\Delta)f)|\\
			=&|\left< f,H(x,t) \right>-\left<\p_{\nu}f,h(x,t)\right>|\\
			\leq& \|f\|_{L^2(Q;e^{-2\lambda^2t})}\|H\|_{L^2(Q;e^{2\lambda^2t}) }+  \|\Delta f\|_{L^2(Q;e^{-2\lambda^2t})}\| u\|_{L^2(Q;e^{2\lambda^2t}) }\\
			&+\|\p_{\nu}f \|_{L^2(\Sigma;e^{-2\lambda^2t})} \|h\|_{L^2(\Sigma;e^{2\lambda^2t}) }\\
			\leq & C \Big[\lambda^{-2}\|H\|_{L^2(Q;e^{2\lambda^2t}) }+\lambda^{-1}\|h\|_{L^2(\Sigma;e^{2\lambda^2t})}+\| u\|_{L^2(Q;e^{2\lambda^2t}) } \Big]\\
			&\times \|( -\p_t-\Delta)f)\|_{L^2(Q;e^{-2\lambda^2t})}.
		\end{aligned}
	\end{equation}
	Notice that $\eqref{define M pre}$ can be rewritten as
\begin{equation}\label{define M pre'}
	\begin{cases}
		-(e^{\lambda^2t}u)_t-\Delta(e^{\lambda^2t}u)+\lambda^2(e^{\lambda^2t}u)-F(\rho e^{\lambda^2t} )= F(x)\theta_{+}e^{- \lambda\mathrm{i}\xi\cdot x} e^{-\mathrm{i}(x,t)\cdot (\eta,\tau)}  & \text{ in } Q,\medskip\\
		(e^{\lambda^2t}u)(x,t)=0   & \text{ on } \Sigma,\medskip\\
		\displaystyle{(e^{\lambda^2t}u)(x,T)=\frac{\delta G}{\delta m}(x,1)\Big( e^{- \lambda\mathrm{i}\xi\cdot x}e^{-\mathrm{i}(x,T)\cdot (\eta,\tau)}\theta_{+}(T) +e^{\lambda^2T}\rho(x,T) \Big)} & \text{ in } \Omega.\\
	\end{cases}
\end{equation}
	Considering the system \eqref{define M pre'} and letting $v= e^{\lambda^2t}u$, we have 
	$ v\in L^2(0,T; H_0^1(\Omega))\cap H^1(0,T;H^{-1}(\Omega))$ and
	\begin{equation}\label{car corollary}
		\begin{aligned}
			\lambda^2 \int_{Q} v^2 dxdt &=\int_{Q} \quad \Big(F(\rho e^{\lambda^2t} )+F(x)\theta_{+}e^{- \lambda\mathrm{i}\xi\cdot x}e^{-\mathrm{i}(x,t)\cdot (\eta,\tau)} \Big)v + v_tv+\Delta v v \quad dxdt\\
			&=\int_{Q} \Big(F(\rho e^{\lambda^2t} )+F(x)\theta_{+}e^{- \lambda\mathrm{i}\xi\cdot x}e^{-\mathrm{i}(x,t)\cdot (\eta,\tau)} \Big)v + \frac{1}{2}\frac{d}{dt}v^2-|\nabla v|^2 dxdt\\
			&\leq \|  F(\rho e^{\lambda^2t} )+F(x)\theta_{+}e^{- \lambda\mathrm{i}\xi\cdot x}   \|_{L^2(Q)}\|v\|_{L^2(Q)}+\frac{1}{2}\int_{\Omega} v(x,T)^2 dx\\
			&\leq C\|\rho\|_{L^2(Q;e^{2\lambda^2t })}.
		\end{aligned}
	\end{equation}
	Hence, 
	\begin{equation}\label{app of u}
		\| u\|_{L^2(Q;e^{2\lambda^2t})}\leq C\lambda^{-2}\|\rho\|_{L^2(Q;e^{2\lambda^2t })}.
	\end{equation}
	It follows that $\mathcal{L}$ is bounded in $X$. By Hahn Banach's extension theorem, $\mathcal{L}$ can be extended to a bounded linear map  on $L^2(Q; e^{-2\lambda^2t})$. Hence, there exisits $p\in L^2(Q; e^{2\lambda^2t})$ such that
	$$\left< f,\Delta u+H(x,t) \right>_0-\left<\p_{\nu}f,h(x,t)\right>_0= \left<( -\p_t-\Delta)f)  ,p \right>_{0}, $$
	and
	\begin{equation}\label{risze}
		\begin{aligned}
			\|p\|_{L^2(Q; e^{2\lambda^2t})} &\leq   C \Big[\lambda^{-2}\|H\|_{L^2(Q;e^{2\lambda^2t}) }+\lambda^{-1}\|h\|_{L^2(\Sigma;e^{2\lambda^2t})}+\| u\|_{L^2(Q;e^{2\lambda^2t})} \Big]\\
			&\leq  C \Big[\lambda^{-1/2}+\lambda^{-1}+\| u\|_{L^2(Q;e^{2\lambda^2t})} \Big].
		\end{aligned}
	\end{equation} 
	Therefore, $p$ is a solution of $\eqref{define M}$ and we define 
	$$\mathcal{M}(\rho)=p.$$
	Now, $\mathcal{M}$ is a map from $J$ to itself.
	
	Let $\rho_1,\rho_2\in J$. By a subtraction, we have
	\begin{equation}
		\begin{cases}
			(\mathcal{M}(\rho_1)- \mathcal{M}(\rho_2))_t-\Delta(\mathcal{M}(\rho_1)- \mathcal{M}(\rho_2))-\Delta (u_1-u_2)=0 & \text{ in } Q,\medskip\\
			(\mathcal{M}(\rho_1)- \mathcal{M}(\rho_2))(x,t)= 0& \text{ on } \Sigma,\\
		\end{cases}
	\end{equation}
	where $u_1-u_2$ satisfies
	\begin{equation}
		\begin{cases}
			-(u_1-u_2)_t-\Delta (u_1-u_2)-F(\rho_1-\rho_2) =0 & \text{ in } Q,\medskip\\
			(u_1-u_2)(x,t)=0   & \text{ on } \Sigma, \medskip\\
			(u_1-u_2)(x,T)=(\rho_1-\rho_2) (x,T))  & \text{ in } \Omega.\\
		\end{cases}
	\end{equation}
	Then by the definition of $\mathcal{M}(\rho_1)$, $\mathcal{M}(\rho_2)$ and simialr arguement above, we have for any $(-\p_t-\Delta)f\in X$, 
	\begin{equation}
		\begin{aligned}
			&|\left< (-\p_t-\Delta)f, \mathcal{M}(\rho_1)- \mathcal{M}(\rho_2)\right>_0|\\
			=&|\left<f,\Delta (u_1-u_2)\right>_0|\\
			\leq&  C \Big(\lambda^{-2}\|\rho_1-\rho_2\|_{L^2(Q; e^{2\lambda^2t}) } \Big)\|( -\p_t-\Delta)f)\|_{L^2(Q;e^{-2\lambda^2t})}.
		\end{aligned}
	\end{equation}
	It follows that
	\begin{equation}\label{final app}	
		\|\mathcal{M}(\rho_1)- \mathcal{M}(\rho_2)\|_{L^2(Q; e^{2\lambda^2t})}		
		\leq  C \Big(\lambda^{-2}\|\rho_1-\rho_2\|_{L^2(Q; e^{2\lambda^2t}) } \Big). 
	\end{equation}
	Therefore, $\mathcal{M}$ is  a contraction mapping  if  $\lambda$  is large enough. By Banach's Fixed Point Theorem, there exists $\rho\in J$ such that $ \mathcal{M}(\rho)=\rho.$  Clearly, it is a solution of the system \eqref{CGO'}.

	Finally, by combining \eqref{risze} and \eqref{app of u}, we have that
	\begin{equation}\label{final app'}
		\|\rho\|_{L^2(Q; e^{2\lambda^2t})}
		\leq C \Big(\lambda^{-1/2}+\lambda^{-1} \Big)\frac{\lambda^2-C}{\lambda^2}.
	\end{equation}
	Therefore, it holds that
	\begin{equation*}
		\begin{aligned}
			&\lim\limits_{\lambda\to\infty} \|w_{+}\|_{L^2(Q)}=\lim\limits_{\lambda\to\infty}\|\rho\|_{L^2(Q; e^{2\lambda^2t})}= 0.
		\end{aligned}
	\end{equation*}
	
	\bigskip

	\noindent (ii).~Since $m\in H_{+}$ and $G\in \mathcal{B}$, we have 
$\frac{\delta G}{\delta m}(x,1)(m(x,T))  \in  L^2(\Omega)$. Hence,
	$u\in L^2(0,T; H_0^1(\Omega))\cap H^1(0,T;H^{-1}(\Omega))$.
	
	The proof is complete. 
\end{proof}

By following similar arguments in the previous theorem, we can derive a similar construction for the adjoint system of $\eqref{CGO}$. In fact, we have

\begin{thm}\label{CGO2}
	Consider the system
	\begin{equation}\label{CGO-}
		\begin{cases}
			u_t-\Delta u= F(x)m & \text{ in } Q, \medskip\\
			-m_t-\Delta m-\Delta u=0  & \text{ in } Q,\medskip\\
			u(x,t)=m(x,t)=0     & \text{ on } \Sigma, \medskip\\
			u(x,0)=  0& \text{ in } \Omega.
		\end{cases}
	\end{equation}
	Assume $F$ is uniformly bounded by $M>0$. Let $\theta_{-}=1-e^{-\lambda^{3/4}(T-t)}$. Then we have that 
	 
	 \begin{enumerate}
	\item[(i)] the system $\eqref{CGO-}$ has a solution in the form 
	$$(u,m)=\Big(u,  e^{\lambda^2t+\lambda\mathrm{i}\xi\cdot x}(\theta_{-}e^{\mathrm{i}(x,t)(\eta,\tau)}+ w_{-}(x,t) )\Big),$$
	such that $\xi,\eta\in\mathbb{S}^{n-1}$, $\xi\cdot\eta=0$, $\tau\in\mathbb{R}$ and 
	\begin{equation}
		\begin{aligned}
			&\lim\limits_{\lambda\to\infty} \|w_{-}\|_{L^2(Q)}=0.
		\end{aligned}
	\end{equation}

\item[(ii)] the solution $(u,m)$ in (i) satisfies $u\in L^2(0,T; H_0^1(\Omega))\cap H^1(0,T;H^{-1}(\Omega))$ and $m\in H_{-}(\Omega).$
\end{enumerate}

\end{thm}

Now we can construct a sequence of solutions $(u,m)\in \big(L^2(0,T; H_0^1(\Omega))\cap H^1(0,T;H^{-1}(\Omega)\big)\times H_{\pm}(Q)$ to the linearized systems. However, we need consider the forward problem in H\"older spaces. So, we need the following approximation property
 in the proof of Theorem $\ref{der F}$ in what follows.
 
\begin{lem}\label{Runge approximation}
	Suppose there is $M>\varepsilon>0$ such that $\varepsilon<F(x)<M$. For any solution $ (u,m)\in  \big(L^2(0,T; H_0^1(\Omega))\cap H^1(0,T;H^{-1}(\Omega)\big)\times H_{+}(Q)$ to 
	\begin{equation}\label{Runge 1}
		\begin{cases}
			-u_t-\Delta u= F(x)m & \text{ in } Q,\medskip\\
			m_t-\Delta m-\Delta u=0  & \text{ in } Q, \medskip\\
			u(x,t)=m(x,t)=0     & \text{ on } \Sigma, \medskip\\
			\displaystyle{u(x,T)=\frac{\delta G}{\delta m}(x,1)(m(x,T) ) }   & \text{ in } \Omega,\\
		\end{cases}
	\end{equation}
	and any $\eta>0$, there exists a solution  $(\hat{u},\hat{m})\in\Big[ C^{1+\frac{\alpha}{2},2+\alpha}(\overline{Q})\Big]^2$ to $\eqref{Runge 1}$ such that
	$$\|m-\hat{m}\|_{L^2(Q)}\leq \eta. $$
\end{lem}
\begin{proof}
	Let $ (u,m)\in  \big(L^2(0,T; H_0^1(\Omega))\cap H^1(0,T;H^{-1}(\Omega)\big)\times H_{+}(Q)$ be a solution to the system $\eqref{Runge 1}.$ Let $\widetilde{m}(x)= m(x,0)$. Then by  Theorem $\ref{construct CGO 1}$, we have $ \widetilde{m}(x)\in H^{-1}(\Omega).$ 
	
	Since $C^{2+\alpha}(\Omega) $ is dense in $H^{-1}(\Omega)$,  for any $\eta>0$, there exists $\widetilde{M}\in C^{2+\alpha}(\Omega)$ such that
	$ \|\widetilde{m}-\widetilde{M} \|_{H^{-1}(\Omega)}\leq\eta.$ Then by Lemma $\ref{linear app unique}$, there is a solution $(\hat{u},\hat{m})\in\Big[ C^{1+\frac{\alpha}{2},2+\alpha}(\overline{Q})\Big]^2$ to the PDE system: 
	\begin{equation}\label{Runge 2}
		\begin{cases}
			-\hat{u}_t-\Delta \hat{u}= F(x)\hat{m} & \text{ in } Q,\medskip\\
			\hat{m}_t-\Delta\hat{m}-\Delta\hat{u}=0  & \text{ in } Q, \medskip\\
			\hat{u}(x,t)=\hat{m}(x,t)=0     & \text{ on } \Sigma, \medskip\\
			\displaystyle{\hat{u}(x,T)=\frac{\delta G}{\delta m}(x,1) (\hat{m}(x,T) )}   & \text{ in } \Omega,\medskip\\
			\hat{m}(x,0)=\widetilde{M}  & \text{ in } \Omega.\\
		\end{cases}
	\end{equation}
	Then we have 
	\begin{equation}\label{Runge 3}
		\begin{cases}
			-(u-\hat{u})_t-\Delta( u-\hat{u})= F(x)(m-\hat{m}) & \text{ in } Q, \medskip\\
			(m-\hat{m})_t-\Delta(m-\hat{m})-\Delta( u-\hat{u})=0  & \text{ in } Q, \medskip\\
			(u-\hat{u})(x,t)=(m-\hat{m})(x,t)=0     & \text{ on } \Sigma, \medskip\\
			\displaystyle{(u-\hat{u})(x,T)=\frac{\delta G}{\delta m}(x,1) ((m-\hat{m})(x,T)) }    & \text{ in } \Omega,\medskip\\
			(m-\hat{m})(x,0)=\widetilde{m}-\widetilde{M}  & \text{ in } \Omega.\\
		\end{cases}
	\end{equation}
	Hence, by Lemma $\ref{prior estimate}$,
	$$ \|m-\hat{m}\|_{L^2(Q)}\leq C \|\widetilde{m}-\widetilde{M} \|_{H^{-1}(\Omega)}\leq C \eta,$$ 
which readily completes the proof. 
\end{proof}

\section{Proof of Theorem ~\ref{der F}}
Before the main proof, we present a key observation as a lemma first. 
\begin{lem}\label{key}
	Let $(v,\rho)$ be a solution of the following system:
	\begin{equation}
		\begin{cases}
			v_t-\Delta v= F(x)\rho & \text{ in } Q,\medskip\\
			-\rho_t-\Delta \rho-\Delta v=0  & \text{ in } Q, \medskip\\
			v(x,t)=\rho(x,t)=0      & \text{ on } \Sigma, \medskip\\
			v(x,0)=  0& \text{ in } \Omega.
		\end{cases}
	\end{equation}
	Let $(\overline{u},\overline{m})$ satisfy
	\begin{equation}
		\begin{cases}
			-\overline{u}_t-\Delta \overline{u}- F_1(x)\overline{m}=(F_1-F_2)m_2  & \text{ in } Q, \medskip\\
			\overline{m}_t-\Delta \overline{m}-\Delta \overline{u}=0  & \text{ in } Q, \medskip\\
			\overline{u}(x,t)=	\overline{m}(x,t)=0    & \text{ on } \Sigma, \medskip\\
			\displaystyle{\overline{u}(x,T)=\frac{\delta G}{\delta m}(x,1)(\overline{m}(x,T))}	   &\text{ in } \Omega, \medskip\\
			\overline{u}(x,0)=0    & \text{ in } \Omega, \medskip\\
			\overline{m}(x,0)=\overline{m}(x,T)=0& \text{ in } \Omega\\
		\end{cases}
	\end{equation}
	Then we have 
	\begin{equation}\label{implies F_1=F_2}
		\int_Q \quad (F_1-F_2)m_2\rho \ dxdt=0.
	\end{equation}
\end{lem}
\begin{proof}
	Notice that
	\begin{equation}\label{IBP1}
		\begin{aligned}
			0&=\int_Q (\overline{m}_t-\Delta \overline{m}-\Delta \overline{u} )\rho\ dxdt\\
			&=\int_{\Omega}\overline{m}\rho\Big|_0^T dx-\int_Q\overline{m}\rho_t dxdt-\int_Q\rho(\Delta\overline{m}+\Delta\overline{u})\ dxdt\\
			&=-\int_Q\overline{m}( -\Delta\rho-\Delta v)dxdt-\int_Q\rho( \Delta\overline{m}+\Delta\overline{u})\ dxdt\\
			&=\int_Q (\overline{m}\Delta v-\rho\Delta\overline{u})\ dxdt.
		\end{aligned}
	\end{equation}
	
	Similarly, one can deduce that
	\begin{equation}\label{IBP2}
		\begin{aligned}
			0&=\int_Q (\overline{m}_t-\Delta \overline{m}-\Delta \overline{u} )v\ dxdt\\
			&=-\int_Q\overline{m}(\Delta v+F_1\rho)dxdt-\int_Q v(\Delta\overline{m}+\Delta\overline{u} )\ dxdt\\
			&=-\int_Q(2\overline{m}\Delta v+F_1\rho\overline{m}+\overline{u} \Delta v)\ dxdt.
		\end{aligned}
	\end{equation}
	Then we have 
	\begin{align*}
		\int_Q (F_1-F_2)m_2\rho dxdt&=\int_Q\quad \rho(-\overline{u}_t-\Delta\overline{u}-F_1\overline{m})\ dxdt\\
		&=\int_{\Omega}\overline{u}\rho\Big|_0^T dx+\int_Q\quad (\rho_t\overline{u}-\rho\Delta\overline{u}-F_1\rho\overline{m})\  dxdt\\
		&=\int_{\Omega}  \overline{u}(x,T)\rho(x,T) dx+ \int_Q\quad (-\Delta\rho-\Delta v)-\rho\Delta\overline{u}-F_1\rho\overline{m}\ dxdt\\
		&=\int_{\Omega} \frac{\delta G}{\delta m}(x,1)(\overline{m}(x,T))\rho(x,T) dx+ \int_Q\quad (-2\rho\Delta\overline{u}-\overline{u}\Delta v-F_1\rho\overline{m})\ dxdt\\
		&=\int_Q\quad (-2\rho\Delta\overline{u}-\overline{u}\Delta v-F_1\rho\overline{m})\ dxdt, 
	\end{align*}
which in combination with $\eqref{IBP1}$ and $\eqref{IBP2}$ readily yields that 
	$$	\int_Q \quad (F_1-F_2)m_2\rho \ dxdt=0.$$
	
	The proof is complete. 
\end{proof}

	
	With all the preparations, we are in a position to present the proof of Theorem~\ref{der F}. 
	
	\begin{proof}[ Proof of Theorem $\ref{der F}$ ]
			For $j=1,2$, let us consider 
		\begin{equation}\label{MFG 1,2}
			\begin{cases}
				-u_t-\Delta u+\frac{1}{2}|\nabla  u|^2= F_j(x,m) & \text{ in } Q',\medskip\\
				m_t-\Delta m-\div (m\nabla u)=0  & \text{ in } Q', \medskip\\
				\p_{\nu}u(x,t)=\p_{\nu}m(x,t)=0     & \text{ on } \Sigma', \medskip\\
				u(x,T)=G(x,m(x,T))     & \text{ in } \Omega',\medskip\\
				m(x,0)=m_0(x) & \text{ in } \Omega'.\\
			\end{cases}
		\end{equation}
		Next, we divide our proof into three steps. 
		
		\bigskip
	\noindent {\bf Step I.}~First, we do the first order linearization to the MFG system \eqref{MFG 1,2} in $Q$ and can derive: 
		\begin{equation}\label{linearization}
			\begin{cases}
				-\p_{t}u^{(1)}_j-\Delta u_j^{(1)}= F_j^{(1)}(x)m_j^{(1)} & \text{ in } Q, \medskip\\
				\p_{t}m^{(1)} _j-\Delta m_j^{(1)} -\Delta u_j^{(1)}=0  & \text{ in } Q, \medskip\\
				\displaystyle{u^{(1)}_j(x,T)=\frac{\delta G}{\delta m}(x,1)(m^{(1)}(x,T))}     & \text{ in } \Omega,\medskip\\
				m^{(1)} _j(x,0)=f_1(x) & \text{ in } \Omega.\\
			\end{cases}
		\end{equation}
		Let $\overline{u}^{(1)}=u^{(1)}_1-u^{(1)}_2$ and $ \overline{m}^{(1)}=m^{(1)} _1-m^{(1)} _2. $ Let 
		$(v,\rho)$ be a solution to the following system
		\begin{equation}\label{adjoint}
			\begin{cases}
				v_t-\Delta v= F^{(1)}_1(x)\rho & \text{ in } Q, \medskip\\
				-\rho_t-\Delta \rho-\Delta v=0  & \text{ in } Q, \medskip\\
				v(x,t)=\rho(x,t)=0      & \text{ on } \Sigma, \medskip\\
				v(x,0)=  0& \text{ in } \Omega.
			\end{cases}
		\end{equation}
	Since $\mathcal{N}_{F_1}=\mathcal{N}_{F_2}$, 
		 by Lemma $\ref{key}$, we have 
		\begin{equation}\label{implies to 0;1}
			\int_Q \quad( F^{(1)}_1(x)-F^{(1)}_2(x))m^{(1)} _2\rho \ dxdt =0,
		\end{equation}
		for all $ m^{(1)} _2\in C^{1+\frac{\alpha}{2},2+\alpha}(Q)$ with $m^{(1)} _2 $ being a solution 
		to $\eqref{linearization}.$
		 By Theorems $\ref{construct CGO 1}$ and $\ref{CGO2}$, there exist $$ m^{(1)} _2(x,t;\lambda)\in H_{+}(Q), \rho(x,t;\lambda)\in H_{-}(Q)$$ in the form
		$$m^{(1)} _2=  e^{-\lambda^2t-\lambda\mathrm{i}\xi_1\cdot x }(\theta_{+}e^{-\mathrm{i}(x,t)(\eta_1,\tau_1)}+ w_{+}(x,t;\lambda) ),$$
		$$\rho=e^{\lambda^2t+\lambda\mathrm{i}\xi_2\cdot x}(\theta_{-}e^{   -\mathrm{i}(x,t)(\eta_2,\tau_2)     }+ w_{-}(x,t;\lambda) ).$$
		Furthermore, we have 
		$$\lim\limits_{\lambda\to\infty}\|w_{\pm}(x,t;\lambda)\|=0.$$

		Now, by Lemma $\ref{Runge approximation}$, there exist a sequence of solutions $\hat{m}_k\in C^{1+\frac{\alpha}{2},2+\alpha}(Q)$ to $\eqref{linearization}$ and 
		$\hat{\rho}_k\in C^{1+\frac{\alpha}{2},2+\alpha}(Q)$ to $\eqref{adjoint}$ such that
		$$\lim\limits_{k\to\infty}\|\hat{m}_k-m^{(1)} _2(x,t;\lambda) \|_{L^2(Q)}=0,$$
		$$\lim\limits_{k\to\infty}\|\hat{\rho}_k-\rho(x,t;\lambda) \|_{L^2(Q)}=0.$$
		Therefore, $\eqref{implies to 0;1}$ implies that
		\begin{equation}\label{implies to 0;2}
			\int_Q \quad( F^{(1)}_1(x)-F^{(1)}_2(x)\hat{m}_k\rho \ dxdt =0,
		\end{equation}
		for all $k\in\mathbb{N}.$
		Let $k,\lambda\to\infty$ in $\eqref{implies to 0;2}$, we have
		\begin{equation}\label{implies to 0;3}
			\int_Q \quad( F^{(1)}_1(x)-F^{(1)}_2(x))e^{-\mathrm{i}(\xi_1+\xi_2,\tau_1+\tau_2)\cdot (x,t)} \ dxdt =0,
		\end{equation}
		for all $\xi_1,\xi_2\in\mathbb{S}^{n-1}$ and $\tau_1,\tau_2\in\mathbb{R}.$ Hence, the Fourier transform of $ ( F^{(1)}_1(x)-F^{(1)}_2(x))$ vanishes in an open set of $\mathbb{R}^n$. Therefore, we have 
		$$ F^{(1)}_1(x)=F^{(1)}_2(x):= F^{(1)}(x)$$
	    in $\Omega.$	
		
		\bigskip
		
		\noindent{\bf Step II.}~We proceed to consider the second linearization to the MFG system $\eqref{MFG 1,2}$ in $Q$ and can obtain for $j=1,2$:
			\begin{equation}
			\begin{cases}
				-\p_tu_j^{(1,2)}-\Delta u_j^{(1,2)}(x,t)+\nabla u_j^{(1)}\cdot \nabla u_j^{(2)}\medskip\\
				\hspace*{3cm}= F^{(1)}(x,t)m_j^{(1,2)}+F^{(2)}(x,t)m_j^{(1)}m_j^{(2)} & \text{ in } \Omega\times(0,T),\medskip\\
				\p_t m_j^{(1,2)}-\Delta m_j^{(1,2)}-\Delta u_j^{(1,2)}= {\rm div} (m_j^{(1)}\nabla u_j^{(2)})+{\rm div}(m_j^{(2)}\nabla u_j^{(1)}) ,&\text{ in } \Omega\times (0,T) \medskip\\
				\displaystyle{u_j^{(1,2)}(x,T)=\frac{\delta G}{\delta m}(x,1)(m_j^{(1,2)}(x,T))+\frac{\delta^2 G}{\delta m^2}(x,1)(m_j^{(1)}m_j^{(2)}(x,T))} & \text{ in } \Omega,\medskip\\
				m_j^{(1,2)}(x,0)=0 & \text{ in } \Omega.\\
			\end{cases}  	
		\end{equation}
	By the proof in Step~I, we have $ (u_1^{(1)},m_1^{(1)})=( u_2^{(1)},m_2^{(1)})$.

	Define $\overline{u}^{(1,2)}=u_1^{(1,2)}-u_2^{(1,2)} $ and $\overline{m}^{(1,2)}=m_1^{(1,2)}-m_2^{(1,2)} $. Since  
	$\mathcal{N}_{F_1}=\mathcal{N}_{F_2}$,  we have
		\begin{equation}
		\begin{cases}
			-\overline{u}^{(1,2)}_t-\Delta \overline{u}^{(1,2)}- F_1(x)\overline{m}^{(1,2)}=(F^{(2)}_1-F^{(2)}_2)m_1^{(1)}m_2^{(1)}   & \text{ in } Q\\
			\overline{m}^{(1,2)}_t-\Delta \overline{m}^{(1,2)}-\Delta \overline{u}^{(1,2)}=0  & \text{ in } Q, \medskip\\
			\overline{u}^{(1,2)}(x,t)=	\overline{m}^{(1,2)}(x,t)=0    & \text{ on } \Sigma, \medskip\\
			\displaystyle{\overline{u}^{(1,2)}(x,T)=\frac{\delta G}{\delta m}(x,1)(\overline{m}^{(1,2)}(x,T))}	   &\text{ in } \Omega, \medskip\\
			\overline{u}^{(1,2)}(x,0)=0    & \text{ in } \Omega, \medskip\\
			\overline{m}^{(1,2)}(x,0)=\overline{m}^{(1,2)}(x,T)=0& \text{ in } \Omega.\\
		\end{cases}
	\end{equation}
Let 
$(v,\rho)$ be a solution to the following system
\begin{equation}
	\begin{cases}
		v_t-\Delta v= F^{(1)}_1(x)\rho & \text{ in } Q, \medskip\\
		-\rho_t-\Delta \rho-\Delta v=0  & \text{ in } Q, \medskip\\
		v(x,t)=\rho(x,t)=0      & \text{ on } \Sigma, \medskip\\
		v(x,0)=  0& \text{ in } \Omega.
	\end{cases}
\end{equation}
By Lemma $\ref{key}$, we have
\begin{equation}
		\int_Q \quad( F^{(2)}_1(x)-F^{(2)}_2(x))m^{(1)} _2m_2^{(1)} \rho \ dxdt =0.
\end{equation}
Next, by a similar argument in the proof of Step~I, we can derive that $( F^{(2)}_1(x)-F^{(2)}_2(x))m^{(1)} _2=0$ for all $m^{(1)} _2 $ as long as it is a solution to $\eqref{linearization}$. We choose $m^{(1)} _2(x,0)\in C^{2+\alpha}(\Omega)$ and it is positive in $\Omega $. Since $m$ satisfies
	\begin{equation}
	\begin{cases}
		\p_{t}m^{(1)} _2-\Delta m_2^{(1)} -\Delta u_2^{(1)}=0  & \text{ in } Q,\medskip\\
		m^{(1)} _2(x,t)=0 & \text{ on } \Sigma, \medskip\\
		m^{(1)} _2(x,0)>0& \text{ in } \Omega,\\
	\end{cases}
\end{equation}
and $u\in C^{1+\frac{\alpha}{2},2+\alpha}(Q)$, we have $ m_2^{(1)}$ cannot be zero in any open set of $Q$. Therefore, we have 
$$ F^{(2)}_1(x)-F^{(2)}_2(x)=0. $$

\bigskip

\noindent{Step~III.}~Finally, by mathematical induction and repeating similar arguments as those in the proofs of Steps I and II, one can show that
$$F^{(k)}_1(x)-F^{(k)}_2(x)=0 ,$$
for all $k\in\mathbb{N}$. Hence, $F_1(x)=F_2(x).$

The proof is complete. 

	\end{proof}

	\section*{Acknowledgements}

The work of was supported by the Hong Kong RGC General Research Funds (projects 12302919, 12301420 and 11300821),  the NSFC/RGC Joint Research Fund (project N\_CityU 101/21), and the France-Hong Kong ANR/RGC Joint Research Grant, A-CityU203/19.

\end{document}